\documentclass[11pt,fleqn]{article}

\usepackage{latexsym}
\usepackage{amsfonts}
\usepackage[leqno]{amsmath}
\usepackage{tikz}
\usepackage{float}
\usepackage{lmodern}

\makeatletter
\newcommand{\leqnomode}{\tagsleft@true}
\newcommand{\reqnomode}{\tagsleft@false}
\makeatother

\newcommand\blackslug{\hbox{\hskip 1pt \vrule width 4pt height 8pt depth 1.5pt
        \hskip 1pt}}

\usepackage{marvosym}
\def\longbox#1{\parbox{0.85\textwidth}{#1}}
\newenvironment{proof}{\noindent {\bf Proof:\ }}{{\quad \blackslug \medbreak}}

\title{Subdivided Claws and the Clique-Stable Set Separation Property }
\author{Maria Chudnovsky\thanks{This material is based upon work supported in part by the U. S. Army
Research Office under grant   number W911NF-16-1-0404, and by  NSF grant DMS-1763817.}
and
Paul Seymour\thanks{Partially supported by NSF grant  DMS-1800053 and AFOSR grant A9550-19-1-0187.}\\
Princeton University, Princeton, NJ 08544}

\date{December 6, 2019; revised \today}

\newtheorem{theorem}{}[section]

\begin{document}
\maketitle

\begin{abstract}
  Let $\mathcal{C}$ be a class of graphs closed under taking induced subgraphs.
  We say that $\mathcal{C}$ has the {\em clique-stable set separation property}
 if there exists $c \in \mathbb{N}$ such that
  for every graph $G \in \mathcal{C}$ there is a
  collection $\mathcal{P}$ of partitions $(X,Y)$ of the vertex set of $G$
  with $|\mathcal{P}| \leq |V(G)|^c$ and with the following property:  
  if $K$ is a clique of $G$, and $S$ is a stable set of $G$, and
  $K \cap S =\emptyset$, then    there
  is $(X,Y) \in \mathcal{P}$ with  $K \subseteq  X$ and $S \subseteq Y$.
  In 1991 M. Yannakakis conjectured that the class of all graphs has the
  clique-stable set separation property, but this conjecture was disproved by
  M. G\"{o}\"{o}s in 2014. Therefore it is now of
  interest to understand for which classes of graphs such a constant $c$ exists. In this paper we define two infinite families $\mathcal{S}, \mathcal{K}$ of
  graphs and show that for every $S \in \mathcal{S}$ and $K \in \mathcal{K}$, the class of graphs with no induced subgraph isomorphic to $S$ or $K$
  has the clique-stable set separation property.

\end{abstract}

\section{Introduction}

All graphs in this paper are finite and simple.  Let $G$ be a graph.   
A {\em clique} in $G$ is a set of pairwise adjacent vertices, and a
{\em stable set} 
is a set of pairwise non-adjacent vertices.
 Let $\mathcal{C}$ be a class of graphs closed under taking induced subgraphs.
  We say that $\mathcal{C}$ has the {\em clique-stable set separation property}
 if there exists $c \in \mathbb{N}$ such that
  for every graph $G \in \mathcal{C}$ there is a
  collection $\mathcal{P}$ of partitions $(X,Y)$ of the vertex set of $G$
  with $|\mathcal{P}| \leq |V(G)|^c$ and with the following property:  
  if $K$ is a clique of $G$, and $S$ is a stable set of $G$, and
  $K \cap S =\emptyset$, then    there
  is $(X,Y) \in \mathcal{P}$ with  $K \subseteq  X$ and $S \subseteq Y$.
  This property plays an important role in a large variety of fields: communication complexity, combinatorial optimization,  constraint satisfaction and others
  (for a comprehensive survey of these connections see \cite{Aureliethesis}).
  
  In 1991 Mihalis Yannakakis conjectured that the class of all graphs has the
  clique-stable set separation property \cite{Yannakakis}, but this conjecture
  was disproved by
  Mika G\"{o}\"{o}s in 2014 \cite{Goos}. Therefore it is now of
  interest to understand for which classes of graphs such a constant $c$ exists; our main result falls into that category.

  Let $G$ be a graph and let $X,Y$ be  disjoint subsets of  $V(G)$.
  We denote by $G[X]$ the subgraph of $G$ induced by $X$,
  by $N(X)$ the set of all vertices of $V(G) \setminus X$ with a neighbor in
  $X$, and by $N[X]$ the set $N(X) \cup X$. 
  We say that $X$ is {\em complete} to $Y$ if every vertex of $X$ is adjacent to
  every vertex of $Y$, and that $X$ is {\em anticomplete} to $Y$ if every
  vertex of $X$ is non-adjacent to every vertex of $Y$. We say that $X$ and $Y$
  are {\em matched} if   every vertex of $X$ has exactly  one
  neighbor in $Y$, and every vertex of $Y$ has exactly one neighbor in $X$ (and therefore $|X|=|Y|$). For a graph $H$, we say that $G$ is {\em $H$-free} if no induced subgraph of $G$ is isomorphic to $H$.

 Next we  define two types of graphs. Let $p,q \in \mathbb{N}$.
  We define the graph $F_S^{p,q}$ as follows:
  \begin{itemize}
  \item $V(F_S^{p,q})=K \cup S_1 \cup S_2 \cup S_3$ where
    $K$ is a clique, $S_1, S_2,S_3$ are stable sets, and the sets
    $K,S_1,S_2,S_3$ are pairwise disjoint;

  \item $|K|=|S_1|=p$, and $K$ and $S_1$ are matched;

  \item $|S_2|=|S_3|=q$, and $S_2$ and $S_3$ are matched;

  \item $K$ is complete to $S_2$;
    
  \item there are no other edges in $F_S^{p,q}$.
  \end{itemize}
  The graph $F_K^{p,q}$ is obtained from $F_S^{p,q}$ by making all pairs of
  vertices of $S_3$ adjacent.

\begin{figure}[H]
\centering

\begin{tikzpicture}[scale=2,auto=left]
\tikzstyle{every node}=[inner sep=1.5pt, fill=black,circle,draw]  
\centering

\node (s11) at (0,2) {};
\node(s12) at (0.5,2){};
\node(s13) at (1,2){};
\node (k1) at (0,1.5) {};
\node(k2) at (0.5,1.4){};
\node(k3) at (1,1.5) {};
\node (s21) at (0,1) {};
\node(s22) at (0.5,1) {};
\node(s23) at (1,1) {};
\node (s31) at (0,0.5) {};
\node(s32) at (0.5,0.5) {};
\node(s33) at (1,0.5) {};

\foreach \from/\to in {k1/k2,k2/k3,k1/k3}
\draw [-] (\from) -- (\to);

\foreach \from/\to in {k1/s11,k2/s12,k3/s13}
\draw [-] (\from) -- (\to);

\foreach \from/\to in {k1/s21,k1/s22,k1/s23,k2/s21,k2/s22,k2/s23, k3/s21,k3/s22,k3/s23}
\draw [-] (\from) -- (\to);

\foreach \from/\to in {s21/s31,s22/s32,s23/s33}
\draw [-] (\from) -- (\to);

\node (t11) at (4,2) {};
\node(t12) at (4.5,2){};
\node(t13) at (5,2){};
\node (l1) at (4,1.5) {};
\node(l2) at (4.5,1.4){};
\node(l3) at (5,1.5) {};
\node (t21) at (4,1) {};
\node(t22) at (4.5,1) {};
\node(t23) at (5,1) {};
\node (t31) at (4,0.5) {};
\node(t32) at (4.5,0.4) {};
\node(t33) at (5,0.5) {};

\foreach \from/\to in {l1/l2,l2/l3,l1/l3}
\draw [-] (\from) -- (\to);

\foreach \from/\to in {l1/t11,l2/t12,l3/t13}
\draw [-] (\from) -- (\to);

\foreach \from/\to in {l1/t21,l1/t22,l1/t23,l2/t21,l2/t22,l2/t23, l3/t21,l3/t22,l3/t23}
\draw [-] (\from) -- (\to);

\foreach \from/\to in {t21/t31,t22/t32,t23/t33}
\draw [-] (\from) -- (\to);

\foreach \from/\to in {t31/t32,t32/t33,t31/t33}
\draw [-] (\from) -- (\to);

\end{tikzpicture}

\caption{The graphs $F_S^{3,3}$ and $F_K^{3,3}$}

\end{figure}
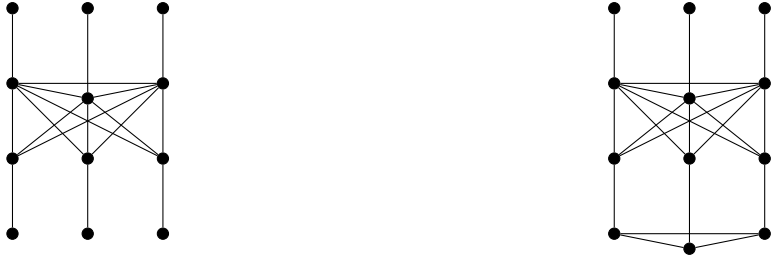
  
  Let $\mathcal{F}^{p,q}$ be the class of all graphs that are both
  $F_S^{p,q}$-free and  $F_K^{p,q}$-free.
  We can now state our main result:
  \begin{theorem}\label{main}
For all $p,q>0$ the class $\mathcal{F}^{p,q}$ has the clique-stable set separation property.
  \end{theorem}
  Since the clique-stable set separation property is preserved under taking
  complements, we immediately deduce:
   \begin{theorem}
     For all $p,q>0$ the class of graphs whose complements are in $\mathcal{F}^{p,q}$ has the
     clique-stable set separation property.
\end{theorem}

     \section{The Proof}

  In this section we prove \ref{main}. The idea of the proof comes from
  \cite{Holes}.   Let $G \in \mathcal{F}^{p,q}$.
  Define $\mathcal{P}_1$ to be the set of all partitions 
  $(N[X], V(G) \setminus N[X])$ and $(N(X), V(G) \setminus N(X))$ where $X$
  is a subset of $V(G)$ with
    $|X| < p$.
  Clearly $|\mathcal{P}_1| \leq 2|V(G)|^p$.
  
  Write $R=R(q,q)$ to mean the smallest positive integer $R$ such that every
  $2$-coloring of the edges of the complete graph on $R$ vertices contains a
  monochromatic complete graph on $q$ vertices.
  Ramsey's Theorem \cite{Ramsey} implies:

  \begin{theorem}\label{Ramsey}
    $R(q,q) \leq 2^{2q}$.
  \end{theorem}
  
  For $a,b \in \mathbb{N}$ let the graph $F_{a,b}$ be defined as follows:
  
\begin{itemize}
  \item $V(F_{a,b})=K_1 \cup S_1 \cup S_2 \cup W$ where
    $K_1$ is a clique, $S_1,S_2$ are  stable sets, and the sets
    $K_1,S_1,S_2,W$ are pairwise disjoint;

  \item $|K_1|=|S_1|=a$, and $K_1$ and $S_1$ are matched;

  \item $|S_2|=|W|=b$, and $S_2$ and $W$ are matched;

  \item $K_1$ is complete to $S_2$;

  \item there is no restriction on the adjacency of pairs of vertices of $W$;
    
  \item there are no other edges in $F_{a,b}$.
\end{itemize}

From the definition of $R$ we  immediately deduce:

\begin{theorem}
  \label{Fab}
  $G$ is $F_{p,R}$-free.
\end{theorem}
For every triple $X=(K_1,S_1,S_2)$ of pairwise disjoint non-emtpy subsets of
$V(G)$ such that $|K_1|=|S_1|=p$ and $|S_2| < R$ we define
the partition $P_X$ of $V(G)$ as follows.
Let $Z$ be the set of all vertices of $G$ that are anticomplete to
$K_1 \cup S_1$. Let $A_X$ be the set of all vertices $v$ of $G$ such that
\begin{itemize}
\item either $v \in K_1$, or $v$ is  complete to $K_1$, and
  \item either
          $v$ has  a neighbor in $S_1$, or
      $v$ has a neighbor in $Z \setminus N(S_2)$.
          \end{itemize}
Note that $A_X$ is disjoint from $S_1 \cup Z$.
Define $P_X=(A_X, V(G) \setminus A_X)$, and let
$\mathcal{P}_2$ be the set of all such partitions $P_X$.
Since  $|K_1 \cup S_1 \cup S_2| \leq  2p+R-1$, and  since by \ref{Ramsey}
$R \leq 2^{2q}$, we deduce that
$|\mathcal{P}_2| < |V(G)|^{2p+2^{2q}}$. 

In order to complete the proof of \ref{main} we will prove the
following:

\begin{theorem}
  \label{P1P2}
  For every clique $K$ and stable set $S$ of $G$ such that $K \cap S = \emptyset$, there exists
  $(X,Y) \in \mathcal{P}_1 \cup \mathcal{P}_2$ with
  $K \subseteq X$ and $S \subseteq Y$.
  \end{theorem}

\begin{proof}
Let $K$ and $S$ be as in the statement of \ref{P1P2}.

\vspace*{-0.4cm}

\begin{equation} \label{maximal}
  \longbox{\emph{We may assume that $K$ is a maximal clique of $G$, and $S$ is a maximal stable set of $G$.}}
\end{equation}

Let $K'$ be a maximal clique of $G$ with $K \subseteq K'$,
and let $S'$ be a maximal stable set of $G$ with $S \subseteq S'$.
If $K' \cap S' = \emptyset$, then the existence of the desired partition for
$K,S$ follows from the existence of such a partition for $K',S'$; thus we
may assume that $K' \cap S' \neq \emptyset$. Since $K'$ is a clique and $S'$ is a
stable set, it follows that $|K' \cap S'|=1$, say $K' \cap S' = \{v\}$.
But now the partitions $(N[\{v\}], V(G) \setminus N[\{v\}])$ and
$(N(\{v\}), V(G) \setminus N(\{v\}])$ are both in  $\mathcal{P}_1$,
  and at least one of them  has the desired property. This proves
  \eqref{maximal}.
\\
\\

In view of \eqref{maximal} from now on we assume that $K$ is a maximal clique
of $G$, and $S$ is a maximal stable set of $G$. Consequently every vertex of $K$
has a neighbor in $S$. Let $S_1' \subseteq S$ be a minimal subset of $S$ such
that every vertex of $K$ has a neighbor in $S_1'$. It follows from the minimality of $S_1'$ that there is a subset $K_1'$ of $K$ such that $S_1'$ and $K_1'$ are
matched. If $|S_1'| < p$, then the partition
$(N(S_1'), V(G) \setminus N(S_1')) \in \mathcal{P}_1$ has the desired property,
so we may assume that $|S_1'| \geq p$.

Let $S_1$ be a subset of $S_1'$ with $|S_1|=p$, and let
$K_1=N(S_1) \cap K_1'$. Then $S_1$ and $K_1$ are matched, and so $|K_1|=p$.
Let $Z$ be the set of vertices of $G$ that are anticomplete to $S_1 \cup  K_1$.
Then $S_1' \setminus S_1 \subseteq Z \cap S$, and in
particular every vertex of $K$ has a neighbor either in $S_1$ or in
$Z \cap S$.
Let $S'$ be the subset of vertices of $S \setminus S_1$ that are complete to
$K_1$. Note that $S' \cap Z = \emptyset$.
Let $S_2$ be a minimal subset of $S'$ such that
$N(S_2) \cap Z = N(S') \cap Z$.
It follows from the minimality of $S_2$ that there is a subset
$W \subseteq Z \cap N(S')$ such that $W$ and $S_2$ are matched. 
Observe that $G[K_1 \cup S_1 \cup S_2 \cup W]$ is isomorphic to
$F_{p,|S_2|}$ (with $K_1,S_1,S_2,W$ as in the definition of $F_{a,b}$). It
follows from \ref{Fab} that $|S_2| < R$. 

Let $X=(K_1,S_1,S_2)$. We claim that the partition $P_X \in \mathcal{P}_2$ has
the desired property
for the pair $K,S$. Recall that $P_X=(A_X, V(G) \setminus A_X)$, where 
$A_X$ is the set of all vertices $v$ of $G$ such that
\begin{itemize}
\item either  $v \in K_1$, or $v$ is  complete to $K_1$, and
  \item either
      $v$ has  a neighbor in $S_1$, or
      $v$ has a neighbor in $Z \setminus N(S_2)$.
    \end{itemize}

We need to show that $K \subseteq A_X$, and $S \cap A_X = \emptyset$.

\vspace*{-0.4cm}\begin{equation} \label{K}
  \longbox{\emph{$K \subseteq A_X$.}}
\end{equation}

Let $k \in K$.
Clearly either $k \in K_1$ or $k$  is complete to $K_1$.
Moreover, $k$ has a neighbor in $S_1'$, and 
$S_1' \subseteq S_1 \cup (Z \cap S)$. Since $S$ is a stable set, 
it follows that  $Z \cap S \subseteq Z \setminus N(S_2)$, and thus
$k$ has a neighbor either in $S_1$, or in $Z \setminus N(S_2)$. This proves
\eqref{K}.

\vspace*{-0.4cm}\begin{equation} \label{S}
  \longbox{\emph{$S \cap A_X = \emptyset$.}}
\end{equation}

Suppose that  $s \in S  \cap A_X$.  Then $s \not \in K_1$; therefore
$s$ is complete to $K_1$, and so $s \in S'$. Since $S$ is a stable set, it follows that $s$ is anticomplete to $S_1$, and therefore $s$ has a neighbor in $Z \setminus N(S_2)$. But
$N(S') \cap Z=N(S_2) \cap Z$, a contradiction. This proves \eqref{S}.
\\
\\

Now \ref{P1P2} follows from \eqref{K} and \eqref{S}.
\end{proof}
This completes the proof of \ref{main}.

\section*{Acknowledgment}
This work was done during the Structural Graph Theory Downunder Matrix Program, that was held in 2019 at the Creswick Campus of the University of Melbourne.
The authors express their gratitude to the Matrix Center for the funding it provided and the  use of its facilities, and to the organizers of the program for the invitation.

\end{document}